\documentclass[11pt]{amsart}

\usepackage[margin=1.5in]{geometry}
\usepackage[utf8]{inputenc}
\usepackage[T1]{fontenc}
\usepackage[spacing,kerning,babel]{microtype}
\usepackage[small,margin=1cm]{caption}

\usepackage{amsmath,amsthm,amssymb,amsfonts,stmaryrd,bbm,mathtools}
\usepackage{graphicx}
\usepackage{hyperref}

\theoremstyle{plain}
\newtheorem{theorem}{Theorem}
\newtheorem{proposition}{Proposition}
\newtheorem{lemma}{Lemma}

\theoremstyle{definition}
\newtheorem{assumption}{Assumption}

\theoremstyle{remark}
\newtheorem{remark}{Remark}

\author{Vincent Beffara}
\date{\today}
\title{Critical percolation on mesoscopic triangulations}

\begin{document}

\maketitle

\begin{abstract}
  We extend Smirnov's proof of  the existence and conformal invariance
  of the scaling limit of  critical site-percolation on the triangular
  lattice  to  particular  sequences  of  periodic  graphs  with  more
  arbitrary  large-scale  structure,   obtained  by  piecing  together
  triangular  regions of  the triangular  lattice. While  not formally
  speaking a scaling  limit statement (as the graphs  are not rescaled
  versions of each  other), the result is a weak  form of universality
  for critical percolation.
\end{abstract}

\section*{Introduction}

The main goal  of statistical physics is to  understand the long-range
properties of  a model  described by local,  microscopic interactions.
Typically, the  model will  be defined  on a lattice  and depend  on a
parameter (usually  denoted by $\beta$  and interpreted as  an inverse
temperature) and exhibit a \emph{phase  transition} at a certain value
$\beta_c$ of the  parameter. This means that  its qualitative behavior
changes  drastically between  the  high-temperature  regime ($\beta  <
\beta_c$), where  decorrelation between  distant regions  is observed,
and  the   low-temperature  regime  ($\beta  >   \beta_c$)  for  which
long-range  ordering   is  present.   The  transition  can   often  be
characterized  by the  behavior of  a \emph{correlation  length} which
depends on the temperature and diverges at the critical point.

Such systems  are most interesting  when considered at  their critical
point. Indeed, the divergence of the correlation length indicates that
if one can define a \emph{scaling  limit} when the mesh of the lattice
is sent to $0$, such a limit should be non-trivial and invariant under
scalings. In two dimensions, many scaling limits are moreover expected
to   be  \emph{conformally   invariant},  \emph{i.e.}   invariant  (or
covariant) under the  action of conformal maps.  An archetypal example
is  that of  Brownian motion,  which is  the scaling  limit of  simple
random walk and is indeed conformally invariant in dimension two.

\subsection{Percolation}
\label{sec-1-1}

There are very few models for  which convergence to a scaling limit is
known.  The  easiest  to describe  is  (Bernoulli)  \emph{percolation}
\cite{grimmett:book,kesten:book}. It is described as follows: starting
with  the  honeycomb lattice  on  the  plane,  color each  face  (each
hexagon) black with probability $p\in(0,1)$ and white with probability
$1-p$, independently  of one  another. One is  then interested  in the
existence  of  unbounded chains  of  adjacent  hexagons that  are  all
colored black. It is not difficult  to show that this is almost surely
the case if $p$ is close enough to $1$, and almost surely not the case
if it is small enough; there  is a critical value $p_c$ separating the
two  possible behaviors.  In  this  particular case,  it  is a  result
similar to Kesten's celebrated theorem \cite{kesten:pc12} that in fact
$p_c=1/2$.

An equivalent  way to  see this model  is to consider  it as  a random
coloring of the  vertices of the triangular lattice, which  is dual to
the honeycomb lattice. The model  itself is therefore usually referred
to as \emph{site percolation}.

\subsection{Cardy's formula}
\label{sec-1-2}

It  is one  of the  most striking  recent results  in the  domain that
critical  site-percolation on  the  triangular lattice  has a  scaling
limit.   One   precise   statement   of   this   uses   \emph{crossing
  probabilities}. Let $\Omega$ be  a simply connected, bounded, smooth
domain in  the complex  plane and  let $(a,b,c,d)$  be a  quadruple of
boundary points  of $\Omega$ ---  these data form  a \emph{topological
  rectangle}. Consider  critical site-percolation on  the intersection
of $\Omega$ with the triangular lattice with mesh $\delta$.

\begin{theorem}       [Smirnov        \cite{smirnov:percol}]       Let
  $C_\delta(\Omega,a,b,c,d)$ be the event that there exists a chain of
  black vertices  in $\Omega$, connecting  the boundary arcs  $ab$ and
  $cd$ --- this is called a \emph{crossing event}. Then,
  \begin{enumerate}
  \item   As    $\delta\to0$,   the    probability   of    the   event
    $C_\delta(\Omega,a,b,c,d)$   converges  to   a  limit   $\mathfrak
    f(\Omega,a,b,c,d) \in (0,1);$
  \item Moreover, the limit is conformally invariant: if $\Phi: \Omega
    \to \Phi(\Omega)$ is a conformal  map that extends continuously to
    the  boundary   of  $\Omega$,  then   \[\mathfrak  f(\Phi(\Omega),
    \Phi(a),     \Phi(b),     \Phi(c),    \Phi(d))     =     \mathfrak
    f(\Omega,a,b,c,d).\]
  \end{enumerate}
\end{theorem}

The  Riemann  mapping  theorem   ensures  that  any  such  topological
rectangle can be  conformally mapped to a rectangle  $\mathfrak R_q :=
[0,q] \times [0,1]$, seen as a subset  of the complex plane, in such a
way that  the boundary  arcs $ab$  and $cd$ are  sent to  the vertical
edges  of $\mathfrak  R_q$;  the value  of $q$  is  prescribed by  the
topological rectangle. The limit $\mathfrak  f$ can then be written in
terms  of $q$:  there is  a function  $\mathfrak C  : (0,+\infty)  \to
(0,1)$   satisfying   \[\mathfrak    f(\Omega,a,b,c,d)   =   \mathfrak
f(\mathfrak  R_q, q,q+i,i,0)  =: \mathfrak  C(q).\] The  existence and
value  of the  function $\mathfrak  C$ were  first predicted  by Cardy
\cite{cardy:formula},  and   the  value  is  known   as  \emph{Cardy's
  formula}.

As a consequence of this theorem,  one gets the existence of a scaling
limit  for macroscopic  percolation  interfaces,  described using  SLE
processes,  from   which  one  can   then  derive  the   existence  of
\emph{critical  exponents}:  for   instance  \cite{smirnov:exps},  the
probability that  the origin is  in an infinite component  scales like
$(p-p_c)^{5/36 + o(1)}$ as $p  \downarrow p_c$, and the probability at
the critical point that the cluster  of the origin has radius at least
$R$ behaves like $R^{-5/48+o(1)}$.

\subsection{Universality}
\label{sec-1-3}

It  is tempting  to  wonder  how general  the  statement of  Smirnov's
theorem really is. The  \emph{qualitative} features of percolation are
the  same  in  many  two-dimensional   lattices:  there  always  is  a
non-trivial critical point, divergence  of the correlation length, and
in  many   cases  uniform   upper  and   lower  bounds   for  crossing
probabilities at  the critical point  (this is usually referred  to as
\emph{Russo-Seymour-Welsh theory} \cite{russo:rsw,seymour:rsw}).

Physicists expect  more quantitative  similarities: the  conjecture of
\emph{universality} states  that all  Bernoulli percolation  models on
planar  lattices   will  exhibit  the  same   critical  exponents  and
essentially  the same  scaling limit  (even  though the  value of  the
critical  point $p_c$  itself depends  on the  lattice). The  question
remains completely open  from the mathematical point of  view, and the
triangular lattice is so far the only  one on which a scaling limit is
known to exist.

\subsection{Statement of the results}
\label{sec-1-4}

The  main  result of  the  present  paper  is a  partial  universality
statement,    for    \emph{mesoscopic   triangulations}    \emph{i.e.}
triangulations  of the  plane with  a local  structure similar  to the
triangular  lattice,  but  with  a  longer  range  one  that  is  more
arbitrary.

Let $\mathcal T$ be a periodic triangulation of the plane, embedded in
the plane  in a fixed way.  Rescale $\mathcal T$ to  have lattice mesh
$\delta$, and replace each of its faces with a triangle of side length
$N$ taken  from the triangular  lattice, welding such  triangles along
the  edges  of the  initial  triangulation.  This produces  a  lattice
$\mathcal  T_{\delta,N}$,   which  has  two   characteristic  lengths,
$\delta$ and $\delta/N$. $\delta$ is understood to be small, but large
with  respect to  the  microscopic scale  $\delta/N$,  hence the  name
\emph{mesoscopic}.

\begin{theorem}
  \label{thm:intro}
  Assuming that  $\mathcal T\!\!$ satisfies assumption  1 (see below),
  there  exists  a  constant  $c>0$ such  that,  as  $\delta\to0$  and
  $N\to\infty$ jointly,  subject to the constraint  $N > \delta^{-c}$,
  critical   percolation   on   the   mesoscopic   lattice   $\mathcal
  T_{\delta,N}$  has  the  same  scaling limit  as  that  of  critical
  percolation  on   the  triangular  lattice,  up   to  a  real-linear
  transformation depending only on  the initial embedding of $\mathcal
  T$.
\end{theorem}

The  precise  way that  the  linear  transformation appearing  in  the
statement is related  to $\mathcal T$ is rather explicit,  and will be
presented  in  detail  below.  In  particular,  one  can  characterize
embeddings for which it is the  identity map, and for such embeddings,
the conclusion of the theorem is that the scaling limit is exactly the
same as for the triangular lattice.

\subsection{Plan of the article}
\label{sec-1-5}

In the next  section, we start by giving some  background and notation
on planar triangulations, subdivisions,  and related Riemann surfaces,
to be able  to give a precise  statement of our main  results. Then in
section~\ref{sec-3}    we    first    investigate    the    case    of
\emph{macroscopic}   triangulations   (\emph{i.e.}   $\delta$   fixed,
$N\to\infty$). In  section~\ref{sec-4} we  extend the argument  to the
mesoscopic case.  In the  section~\ref{sec-5}, we  make a  few remarks
about  assumption 1  and possible  relaxations of  it. Two  appendices
contain a characterization of quasi-conformal maps similar to Morera's
theorem, and a few details about computational aspects.

\section{Riemann surfaces from discrete structures}
\label{sec-2}

\subsection{Graphs on the torus}
\label{sec-2-1}

Let $\mathbb  T$ be the flat  torus $\mathbb R^2 /  \mathbb Z^2$, seen
for the  moment as  a topological  space. Let $G$  be a  finite graph,
embedded in $\mathbb T$ in such a  way that its edges do not intersect
away from their endpoints and that its complement is composed of faces
homeomorphic to  disks. We will  always implicitly assume that  $G$ is
\emph{simple} (in the sense that there is at most one edge between any
pair of  vertices) and has  no \emph{loops} (\emph{i.e.} edges  from a
vertex to itself).

Given such a graph, the \emph{degree} of a vertex is simply the number
of neighbors it has in $G$, and the  degree of a face is the number of
edge sides it has on its boundary  (the same edge can count twice, for
instance if $G$ has a leaf). If  all the faces have degree $3$, $G$ is
called a \emph{triangulation} of the torus.

Given  a triangulation  of  the  torus, one  can  construct a  Riemann
surface  of genus  $1$ by  ``gluing together  equilateral triangles''.
More formally, one  can give $\mathbb T$ a  complex manifold structure
by  fixing  a homeomorphism  between  each  face and  the  equilateral
triangle of  vertices $(1,e^{2\pi i/3},e^{-2\pi i/3})$  in the complex
plane  and extending  it to  a proper  atlas by  using the  reflection
principle  along the  edges of  the triangulation.  We will  denote by
$M_G$ the Riemann surface obtained this way.

\begin{remark}
  There are  of course many other  ways to create a  complex structure
  out of  a discrete  one (circle packings  and branched  coverings to
  cite  two;  see   \emph{e.g.}~\cite{LZ:graphs}).  In  general,  they
  provide a \emph{different} Riemann surface, which is not conformally
  equivalent to $M_G$. While the one  we chose here might seem like an
  arbitrary choice, our  result would be the same  for most reasonable
  choices; indeed, asymptotically  as the triangles are  more and more
  refined, the difference between various constructions then vanishes.
\end{remark}

\subsection{Construction \emph{via} quasi-conformal maps}
\label{sec-2-2}

\begin{figure}[ht!]
  \centering
  \includegraphics[width=\hsize]{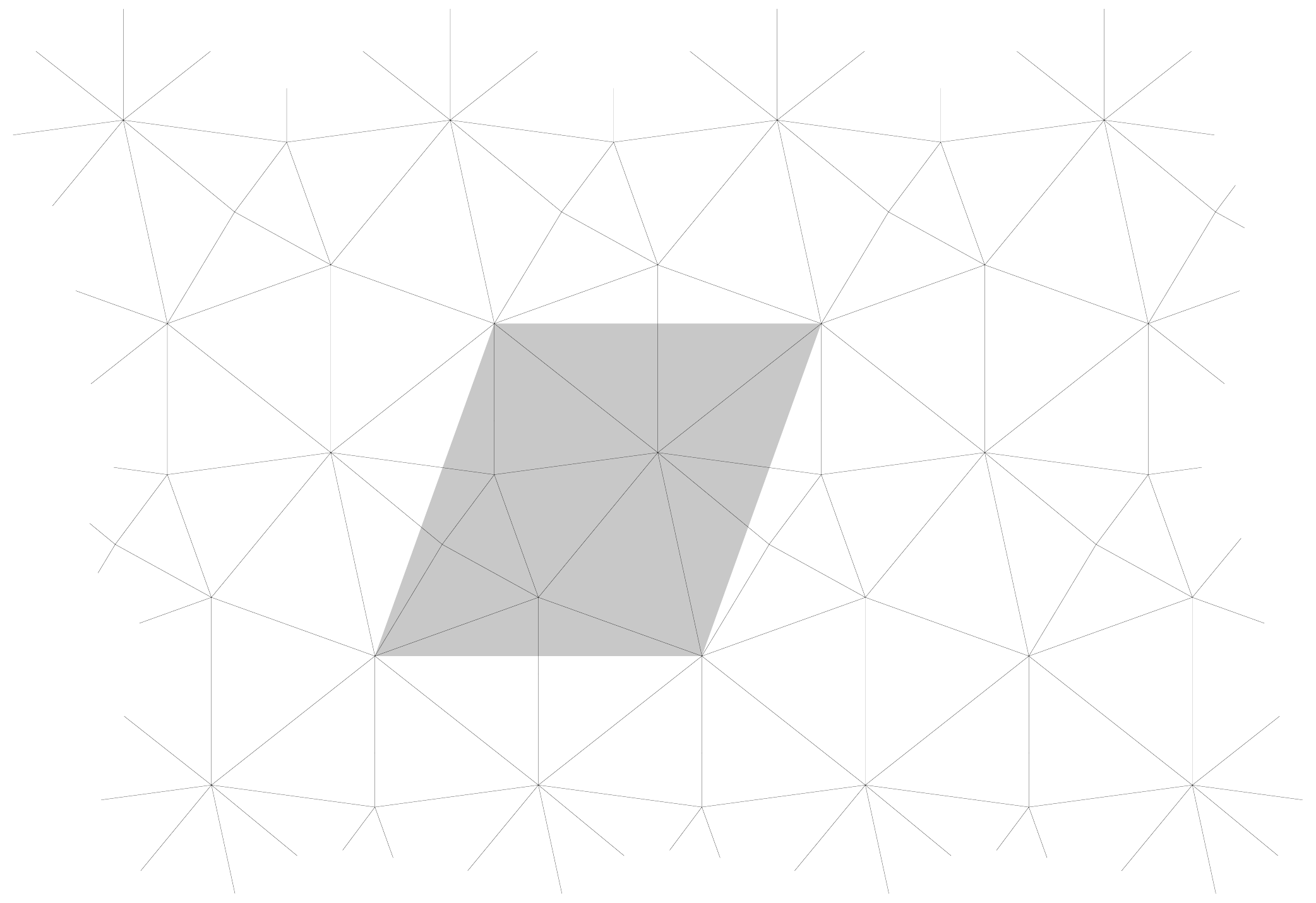}
  \caption{A  periodic  triangulation,  which   can  be  seen  as  the
    universal cover  of a triangulation of  the torus. In gray  is one
    fundamental domain.}
  \label{fig:tri}
\end{figure}

A (perhaps)  more explicit construction  of the Riemann  surface $M_G$
associated to a  triangulation goes as follows. Pick  any embedding of
$G$  in $\mathbb  T$  such  that edges  are  mapped  to straight  line
segments. Then, lift  the embedding to the whole  plane $\mathbb R^2$,
to get  a doubly  periodic triangulation of  it. This  doubly periodic
graph has a fundamental domain, which is a parallelogram; the shape of
this parallelogram can be chosen  arbitrarily (it correspond to giving
an additional structure to the torus $\mathbb T$), but for convenience
assume  that one  of its  edges is  the unit  segment along  the first
coordinate axis.

Identify $\mathbb R^2$ to the complex  plane $\mathbb C$; and for each
triangular  face $f$  thus obtained,  let
\begin{equation}
  \mu_f := - \frac {a+\tau b+\tau^2 c}  {\bar a + \tau \bar b + \tau^2
    \bar c}
  \label{eq:mu_f}
\end{equation}
(where  $a$, $b$  and $c$  are  the vertices  of $f$,  ordered in  the
positive direction  around the face  and seen as complex  numbers, and
where $\tau = e^{2\pi i/3}$).

Notice that  $\mu_f$ vanishes  if and  only if  $f$ is  an equilateral
triangle; and  that it  is invariant under  cyclic relabelling  of the
vertices,  so   that  it   is  well-defined.  The   strict  inequality
$|\mu_f|<1$  is  satisfied  as  soon  as none  of  the  faces  of  the
triangulation is degenerate, which we will implicitly assume. In fact,
it is  easy to check that  if $\varphi_f$ is any  real-linear function
mapping $f$ into an  equilateral triangle, and preserving orientation,
then it  has constant Beltrami coefficient
\begin{equation}
  \frac {\partial_{\bar z} \varphi_f} {\partial_z \varphi_f} = \mu_f.
\end{equation}

\begin{figure}[ht!]
  \centering
  \includegraphics[width=\hsize]{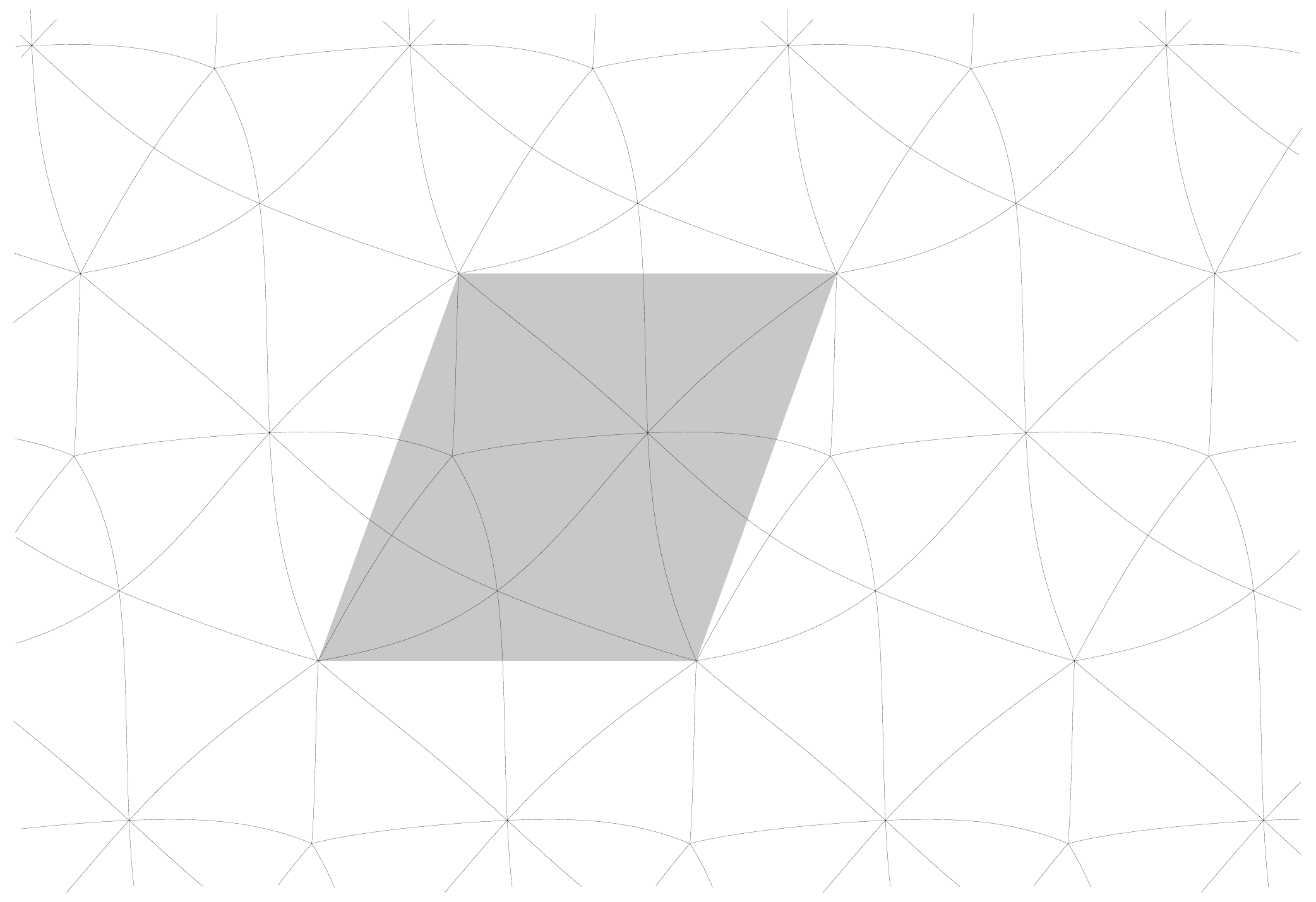}
  \caption{The image  of the triangulation of  Figure~\ref{fig:tri} by
    the corresponding bi-periodic quasi-conformal map $\varphi_G$. The
    gray  fundamental  domain  has  vertices $0$,  $1$,  $\tau_G$  and
    $1+\tau_G$.}
  \label{fig:tri_qc}
\end{figure}

Now, let $\mu  : \mathbb C \to  \mathbb C$ be equal to  $\mu_f$ in the
interior of each face $f$, and to  $0$ along edges and at vertices. By
periodicity, $\|\mu\|_\infty<1$, so by  the measurable Riemann mapping
theorem  (for this  and  more background  on  quasi-conformal maps  in
general,  cf.~\cite{ahlfors:qc}) there  exists  a quasi-conformal  map
$\varphi_G : \mathbb C \to \mathbb C$ with Beltrami coefficient $\mu$.
This map  is unique if  one adds the conditions  that $\varphi_G(0)=0$
and $\varphi_G(1)=1$.  Moreover, it  preserves the periodicity  of the
initial lattice: there exists a  complex number $\tau_G \in \mathbb H$
such that for any integers $x$ and $y$, and any complex $z$,
\begin{equation}
  \varphi_G (z+x+iy) = \varphi_G (z) + x + \tau_G y.
\end{equation}
What this implies is that  $\varphi_G$ passes to the quotient, leading
to a map
\begin{equation}
  \Phi_G :  \mathbb R^2  / \mathbb Z^2  \to \mathbb C  / (\mathbb  Z +
  \tau_G \mathbb Z) =: \mathbb T_G.
\end{equation}

The way that  $\mu_f$ is chosen, corresponding to the  linear map from
$f$ to an equilateral triangle, directly implies the following result,
which we state as a lemma for easier reference:

\begin{lemma}
  \label{lem:qc}
  $\mathbb T_G$ is conformally equivalent to $M_G$.
\end{lemma}

One can also look at the image of the initial triangulation by the map
$\varphi_G$ (see  Figure~\ref{fig:tri_qc}). This provides  a canonical
embedding of $G$ in  the plane; we will see later that  it is for this
choice  of embedding  that the  real-linear  map in  the statement  of
Theorem~\ref{thm:intro}   it  the   identity.  Notice   that  if   the
triangulation is embedded with this  fundamental domain to start with,
then the difference $\varphi_G(z)-z$ is uniformly bounded.

\subsection{Subdivisions}
\label{sec-2-3}

Let  $T =  (V,E)$ be  a  triangulation of  the torus.  We construct  a
subdivision $T'$  of $T$, as follows  (see Figure~\ref{fig:subd}). The
set of  vertices of $T'$  is $V(T') = V  \cup E$; two  vertices $v_1$,
$v_2\in V(T')$ are adjacent if and only if
\begin{itemize}
\item one  is a  vertex of $T$  and the  other one is  an edge  of $T$
  incident to that vertex; or,
\item both are edges of $T$, incident  to a common vertex and a common
  face.
\end{itemize}

\begin{figure}[ht!]
  \centering
  \includegraphics[width=\hsize]{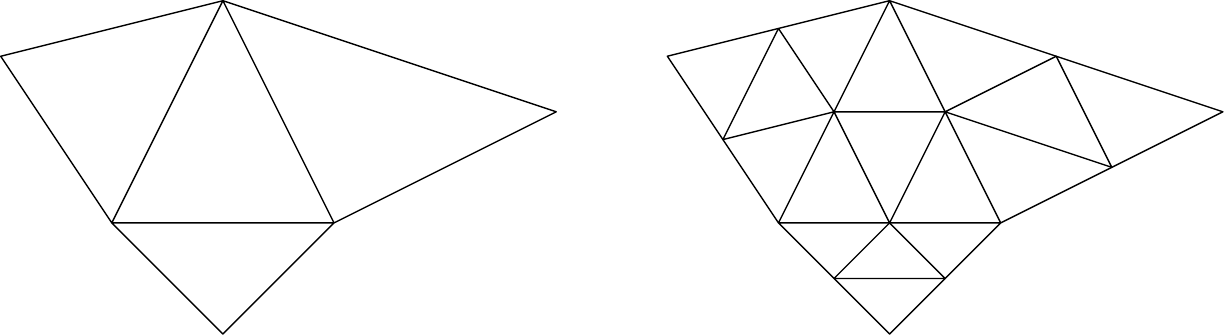}
  \caption{Subdivision of  a triangulation (left: four  adjacent faces
    of $T$; right: the corresponding faces of $T'$)}
  \label{fig:subd}
\end{figure}

More intuitively,  this amounts to replacing  each face of $T$  by its
subdivision into $4$ triangles, with a new vertex located on each edge
of   $T$.    This   construction    can   then   be    iterated   (see
Figure~\ref{fig:subd_n});   we   will    denote   by   $T^{(n)}$   the
triangulation obtained from $T$ after $n$ successive subdivisions. The
operation of subdivision is \emph{a  priori} combinatorial, but in the
case of embedded triangulation, we will always embed the result of the
subdivision in such a way that every  new vertex is at the midpoint of
the  corresponding (embedded)  edge.  This in  turn  implies that  the
length of every  edge of the subdivided  graph is one half  of that of
some edge of the initial one.

A key remark  for what follows is that the  structure of $T^{(n)}$ can
be described  in a  simple way starting  from that of  $T$: it  can be
obtained by replacing each face of  $T$ with a large triangular region
of side length $2^n$ of the regular triangular lattice. In particular,
all  the  newly  added  vertices   have  degree  equal  to  $6$.  More
importantly, the complex structure  introduced in the previous section
is not affected by the operation:

\begin{figure}[ht!]
  \centering
  \includegraphics[width=\hsize]{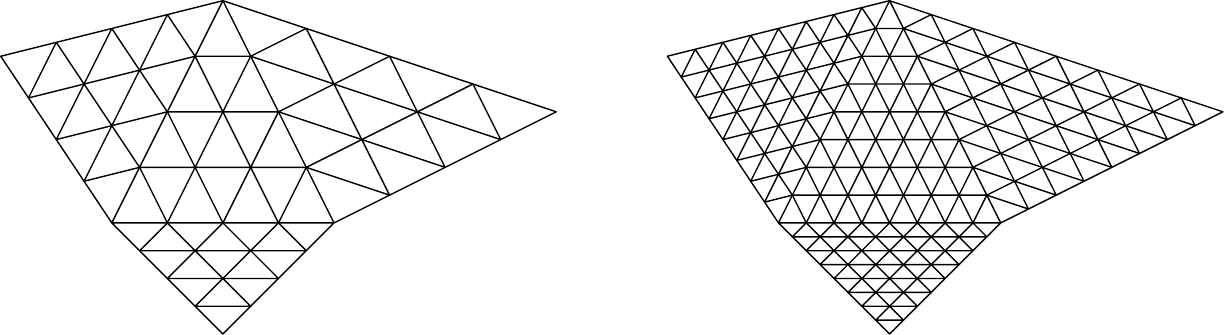}
  \caption{Subsequent   subdivisions    of   the    triangulation   in
    Figure~\ref{fig:subd}}
  \label{fig:subd_n}
\end{figure}

\begin{lemma}
  The Riemann surfaces $M_T$ and $M_{T'}$ are conformally equivalent.
\end{lemma}

\begin{proof}
  This  is clear  from  the  construction of  $M_T$  and $M_{T'}$;  an
  explicit conformal map  between the two can be  constructed from the
  map $z  \mapsto 2z$  in the  complex plane.  Alternatively, noticing
  that all the triangles contained in one of the original faces of the
  triangulation  have the  same shape,  and  thus the  same value  for
  $\mu_f$,  the  conformal  equivalence  is a  direct  consequence  of
  Lemma~\ref{lem:qc}.
\end{proof}

\begin{remark}
  The fact  that the subdivisions  that we  define here are  dyadic is
  mainly  for  ease  of  notation;  one could  as  well  replace  each
  triangular face of  $T$ with a triangular region of  side length $N$
  for other values  of $N$. This will be implied  when considering the
  refined triangulation $\mathcal T_{\delta,N}$.
\end{remark}

\subsection{Other topologies}
\label{sec-2-4}

The same  constructions as above  can be performed in  other settings;
most useful for  us, but postponed to here because  they necessitate a
little  more  notation,  is  the  case  of  approximations  of  simply
connected  domain in  the  complex plane,  or  subdivisions of  planar
triangulations with a boundary.

Let $G$ be a planar graph with  one marked face, embedded in the plane
in such a way that the marked  face is the only unbounded one; we will
refer to  that face as  the outer face of  $G$. Assume that  all inner
faces of $G$ are triangles; to  avoid degenerate cases, we will always
also assume that the outer face of $G$ has at least $4$ vertices. This
last restriction is not an essential  one: we are going to investigate
finer and  finer subdivisions  of a  fixed graph,  and after  one such
subdivision the outer  face cannot have less than $6$  vertices in any
case,  so up  to using  the  result of  the first  subdivision as  the
initial graph, the condition can be assumed to be satisfied.

By   welding  together   equilateral   triangles   according  to   the
combinatorics of the  inner faces of $G$, one can  construct a complex
structure in the  complement of the outer face.  Accordingly, to every
choice of  $4$ boundary vertices  $(a,b,c,d)$ of $G$  with appropriate
ordering, there  corresponds a unique modulus  $\rho = \rho_G(a,b,c,d)
\in \mathbb  R_+$ and a  unique conformal map  $\Psi_{G;a,b,c,d}$ from
the  complement  of  the  outer   face  (equipped  with  this  complex
structure) to  the rectangle  $\mathfrak R_\rho  = [0,\rho_G(a,b,c,d)]
\times [0,1] \subset \mathbb R^2 \simeq \mathbb C$, mapping $a$ (resp.
$b$, $c$, $d$) to $\rho$ (resp. $\rho + i$, $i$, $0$).

As  before,  the  modulus  $\rho_G(a,b,c,d)$ is  invariant  under  the
subdivision      described      earlier:
\begin{equation}
  \rho_G(a,b,c,d) = \rho_{G'}(a,b,c,d).
\end{equation}
In addition,  it is easy  to check that under  successive refinements,
the   mesh    of   the    image   of    $G^{(n)}$   under    $\Psi   =
\Psi_{G^{(n)};a,b,c,d}$, \emph{i.e.} the largest diameter of the image
of a  face of $G^{(n)}$, goes  to $0$ exponentially fast  in $n$. That
makes the setup right for the  study of scaling limits. Yet as before,
seen as a map between domains of  the complex plane, the map $\Psi$ is
quasi-conformal and its  Beltrami derivative is constant  on each face
of the triangulation, given by the same formula as above.

We now  have all the necessary  notation to precisely state  our first
main result. Recall that $\mathfrak  C(\rho)$ is the limiting crossing
probability for  critical site-percolation on the  triangular lattice,
as given by Cardy's formula.

\begin{theorem}
  \label{thm:cardy}
  With  the previous  setup,  consider  site-percolation at  parameter
  $1/2$ on the vertices of $G^{(n)}$,  and let $C_n$ be the event that
  there  exists  a path  of  open  vertices  crossing it  between  the
  boundary arcs  $(ab)$ and $(cd)$.  Then as $n \to  \infty$, \[P[C_n]
  \to \mathfrak C(\rho_G(a,b,c,d)).\]
\end{theorem}

\section{Macroscopic triangulations}
\label{sec-3}

\subsection{Framework of the proof}
\label{sec-3-1}

We first focus on the  \emph{macroscopic} case, \emph{i.e.} for now we
consider   percolation   on   successive  refinements   of   a   fixed
triangulation of a planar domain (in other words, $\delta$ is fixed in
this section). Theorem~\ref{thm:cardy} above is stated in this setup.

One possible idea  to prove Theorem~\ref{thm:cardy} would  be to first
focus at  the contents of  an individual face,  and state that  in the
scaling  limit  one  gets  a continuous  object  described  using  for
instance $CLE_6$.  The restriction of  $G^{(n)}$ to  a face of  $G$ is
indeed an affine  deformation of the usual triangular  lattice, so the
scaling limit of percolation on  it, while deformed accordingly, would
still be  essentially the same as  that of percolation on  the regular
case.

The main  difficulty would  then be to  somehow connect  these scaling
limits across edges of $G$, in a precise enough way to obtain an exact
limit.  Unfortunately, this  seems to  be extremely  difficult if  not
impossible to do with this level  of precision. Writing $G$ as a union
of overlapping lozenges (intersecting on  faces) goes some way in this
direction but seems to be insufficient to get explicit limits. Proving
convergence directly at the level of the exploration process first and
deriving crossing  probabilities from it  is actually doable,  but the
direct approach that we apply instead is more robust in the sense that
it will readily extend to the mesoscopic case.

\subsection{Setup and notation}
\label{sec-3-2}

We actually follow  Smirnov's original proof quite  closely, and refer
the    reader    to     the    articles    \cite{smirnov:perco}    and
\cite{beffara:easy} for a  more detailed version of some  of the steps
below;  rather  than  reproducing  every argument,  we  point  to  the
relevant differences as we go. The argument has two main steps:
\begin{itemize}
\item first, show a compactness result  in order to be able to extract
  subsequential limits of suitable quantities;
\item  then  identify  the  unique possible  subsequential  limit,  by
  showing a conformal invariance property.
\end{itemize}
The choice  of the  quantity of  interest, and the  first step  of the
proof,  are  extremely  similar  to their  counterparts  in  Smirnov's
article  \cite{smirnov:percol}.  Only  the  identification  is  really
different.

Let $\Omega$ be the simply  connected domain in the plane triangulated
by $G^{(n)}$. If $f$ is a face of $G^{(n)}$ (or equivalently, a vertex
of its dual graph), let $E_a^{(n)}(f)$  be the event that there exists
a chain of  pairwise distinct black vertices of  $G^{(n)}$ joining two
boundary points  of $\Omega$ and  separating it into two  regions, one
containing $a$ and  $f$ and the other one containing  $b$ and $c$. Let
$H_a^{(n)}$ be the  probability of that event;  define $H_b^{(n)}$ and
$H_c^{(n)}$ similarly,  in the  obvious fashion.  Recall that  $\tau =
e^{2\pi i/3}$,  and define \[H^{(n)}  := H_a^{(n)} + \tau  H_b^{(n)} +
\tau^2 H_c^{(n)}\] (considered as a function from $\Omega$ to $\mathbb
C$ that is  constant on the faces of $G^{(n)}$).  The main result will
be a consequence of the following:

\begin{proposition}
  As  $n\to\infty$,  $H^{(n)}$  converges   uniformly  to  the  unique
  quasi-conformal map  $h:\Omega\to\mathbb C$  which maps  $\Omega$ to
  the equilateral triangle with vertices $(1,\tau,\tau^2)$, $a$ (resp.
  $b$,  $c$)  to  $1$  (resp. $\tau$,  $\tau^2$),  and  with  Beltrami
  coefficient $\mu$ chosen according to~\eqref{eq:mu_f}.
\end{proposition}

Indeed, once this  proposition is known to hold, it  is enough to look
at the observable $H^{(n)}$, and at its limit $h$, at a boundary point
$d$ to obtain the value of the corresponding crossing probability.

\subsection{Compactness and RSW estimates}
\label{sec-3-3}

The first  argument is  one of  uniform continuity.  We first  state a
result similar to the classical Russo-Seymour-Welsh estimate:

\begin{proposition}
For  every $\lambda>1$  there  exists $c>0$  such  that the  following
holds. Let  $\mathfrak R$  be a rectangle  of aspect  ratio $\lambda$,
entirely contained in $\Omega$. Then,  for every $n$ large enough, the
probability that percolation at  parameter $1/2$ on $G^{(n)}$ contains
a  path  of  black  vertices  crossing  $\mathfrak  R$  lengthwise  is
contained in $[c,1-c]$.
\end{proposition}

\begin{proof}
  The  result is  well-known in  the  case of  the regular  triangular
  lattice. In particular,  there is nothing to prove  if the rectangle
  $\mathfrak R$ is  entirely contained within one of the  faces of the
  initial triangulation  $G$. We  give a brief  explanation of  how to
  extend it to the general case.

  Let  $\mathfrak R$  be a  fixed  rectangle. There  exists a  thinner
  rectangle $\mathfrak R'$  contained in $\mathfrak R$,  such that any
  crossing of  $\mathfrak R'$  also crosses  $\mathfrak R$,  and which
  contains no vertex  of $G$ and intersects a certain  number of edges
  of  $G$ transversally.  Now, between  each pair  of successive  such
  edges, $\mathfrak R'$  intersects $G^{(n)}$ as a piece  of the usual
  triangular lattice, so each of these ``sections'' is crossed in both
  directions with positive  probability. On the other  hand, the union
  of two  successive sections is  still isomorphic,  as a graph,  to a
  piece  of the  triangular lattice,  so every  such union  is crossed
  lengthwise with positive probability.  By the Harris-FKG inequality,
  all   of  these   crossings  exist   simultaneously  with   positive
  probability, and  when they do,  their union contains a  crossing of
  $\mathfrak R'$ and hence of $\mathfrak R$.

  It is easy to obtain uniformity  in the estimate, for a given aspect
  ratio of $\mathfrak R$, because the  graph $G$ is finite --- this is
  where the mesoscopic case will be  more involved. For instance, by a
  pigeonhole argument,  $\mathfrak R'$ can  be chosen with a  width at
  least equal to $|V(G)|^{-1}$ times that of $\mathfrak R$.
\end{proof}

\begin{proposition}
  There exist two constants  $\varepsilon>0$, $K<\infty$ such that the
  following holds.  Let $n>0$;  let $f_1$  and $f_2$  be two  faces of
  $G^{(n)}$,  and denote  by $d(f_1,f_2)$  the distance  between them,
  either  in  the  Hausdorff  sense,  or  equivalently  the  Euclidean
  distance   between   their    centers.   Then,   \[|H^{(n)}(f_2)   -
  H^{(n)}(f_1)| \leqslant  K d(f_1,f_2)^\varepsilon.\]  In particular,
  the family  $(H^{(n)})$ is  relatively compact  for the  topology of
  uniform convergence: any sequence $n_k \to \infty$ has a subsequence
  $(n_{j(k)})$ along which $H^{(n_{j(k)})}$  converges uniformly to an
  $\varepsilon$-Hölder function defined on $\Omega$.
\end{proposition}

\begin{proof}
  The core of the  argument is the same as in  Smirnov's article: if a
  self-avoiding  circuit of  black  vertices surround  both $f_1$  and
  $f_2$ then  either the events $E_a^{(n)}(f_1)$  and $E_a^{(n)}(f_2)$
  are both realized, or none  of them is. Consequently, the difference
  in the  statement of the lemma  is bounded above by  the probability
  that no such circuit exists. But this last probability is at most of
  polynomial order in the distance between the two faces, which can be
  proved using RSW estimates.
\end{proof}

\subsection{Identification of the limit}
\label{sec-3-4}

The central result of this section is the following:

\begin{proposition}
  Let $h$ be any subsequential limit of the sequence $(h_n)$: then $h$
  is quasi-conformal  on $\Omega$,  and on  each face  $f$ of  $G$ its
  Beltrami coefficient is a.e.\ equal to $\mu_f$.
\end{proposition}

\begin{proof}
  The argument  is actually rather simple.  Let $f$ be a  face of $G$.
  There exists  a continuous,  one-to-one map  $\sigma_f :  \Omega \to
  \mathbb C$ which is real-linear within each face of $G$ and maps $f$
  itself to an  equilateral triangle. This map  is quasi-conformal and
  its Beltrami  coefficient in $f$  is equal  to $\mu_f$. We  will use
  $\sigma_f$ to  look at percolation ``from  the point of view  of the
  face $f$''.

  For $n>0$, let $h_n^\sigma =  h_n \circ \sigma_f^{-1}$ be defined on
  $\sigma_f(\Omega)$.  Because  the  definition  of  $h_n$  is  purely
  combinatorial in nature, $h_n^\sigma$ is exactly the observable that
  one would  get if $G$ were  embedded as its image  by $\sigma_f$. In
  particular,  it shares  all  its local  features,  most notably  the
  color-switching lemma.  Moreover, convergence is conserved  by right
  composition with  $\sigma_f^{-1}$: fix a subsequential  limit $h$ of
  $(h_n)$  and  let  $h^\sigma  = h  \circ  \sigma^{-1}$,  defined  on
  $\sigma_f(\Omega_G)$ as well.

  The  key  remark  is  now  that the  initial  proof  of  Smirnov  is
  fundamentally local:  copying it  \emph{mutatis mutandis},  one gets
  that the  contour integral of  $h^\sigma$ along each  closed contour
  \emph{contained   in  $\sigma_f(f)$}   vanishes.  This   means  that
  $h^\sigma$ is holomorphic in $\sigma_f(f)$,  and in turn that $h$ is
  quasi-conformal on $f$ and has the same Beltrami derivative on it as
  $\sigma_f$.
\end{proof}

Now,  let  $h$ be  a  subsequential  limit  of $(h_n)$.  Remember  the
definition of  the map $\Psi$ in  the previous section; it  solves the
same Beltrami  equation as $h$, which  means that they are  related to
one another  by left-composition by  a conformal map. In  other words,
the map $g := h \circ \Psi^{-1}$ is a holomorphic map on the rectangle
$[0,\rho] \times  [0,1]$. The rest  of the discussion is  then exactly
the same as  in the regular case:  $g$ has its image  contained in the
triangle with  vertices $(1,\tau,\tau^2)$,  maps boundary  to boundary
homeomorphically,  and  from  known  boundary  values,  one  can  then
conclude that  $g$ is the unique  conformal map from the  rectangle to
the equilateral  triangle with  these properties, which  concludes the
proof.

\section{Mesoscopic triangulations}
\label{sec-4}

We now  turn to  the proof  of convergence in  the case  of mesoscopic
triangulations,  \emph{i.e.}  as  both $\delta\to0$  and  $N\to\infty$
simultaneously. The overall  framework of the argument is  the same as
before, but more care is needed in order to control the convergence to
$0$  of contour  integrals,  and in  addition the  Russo-Seymour-Welsh
estimate is not a direct consequence of the classical one.

Intuitively, the slower $N$ increases relative to $\delta$, the harder
the proof  becomes --- and indeed  the bounded $N$ situation  would be
full  universality for  percolation on  triangulations, which  is very
much beyond  reach. In fact,  it is  far from clear  whether embedding
using  welded equilateral  triangles  remains relevant  in that  case,
although it  is certainly the  right thing to do  as soon as  $N$ does
tend to infinity.

\subsection{Main assumption: uniform RSW estimates}
\label{sec-4-1}

We  begin  as  before  with \emph{a  priori}  estimates  for  crossing
probability.  Consider the  lattice $\mathcal  T_{\delta,N}$ on  which
site-percolation for  parameter $1/2$  is defined. Moreover,  choose a
rectangle $\mathfrak R$ of aspect  ratio $\lambda>1$. We will from now
on assume that the following holds.

\begin{assumption}
  There exists  a constant  $c=c(\lambda)$ depending on  $\lambda$ but
  not  on the  size or  the orientation  of $\mathfrak  R$ such  that,
  uniformly  as $\delta\to0$  and $N\to\infty$,  the probability  that
  $\mathfrak  R$ is  crossed in  the  long direction  is contained  in
  $[c,1-c]$.
\end{assumption}

It is  equivalent to make the  assumption for one particular  value of
$\lambda>1$ and to  make it for all $\lambda>1$,  because crossings of
longer rectangles can  easily be constructed from  unions of crossings
of shorter ones. In the previous section, \emph{a priori} estimates on
crossing probabilities were used  to obtain the uniform equicontinuity
of the observable as the lattice gets finer and finer: this will still
be the  case here;  we refer the  reader to the  last section  of this
paper for more remarks about this assumption.

\subsection{The setup of the proof}
\label{sec-4-2}

Let $\Omega$ be a simply connected domain of the plane with $3$ points
$a$, $b$  and $c$ on its  boundary, in positive order;  fix $\delta>0$
and $N>0$ for now (with the understanding that $\delta$ will go to $0$
and   simultaneously  $N$   will  go   to  infinity).   Let  $\mathcal
T_{\delta,N}$ be the triangulation of the plane described in the first
introduction; we will call \emph{cells}  the faces of the rescaled but
not yet subdivided  triangulation. The length of an  edge of $\mathcal
T_{\delta,N}$ is of  order $\delta/N$, and the diameter of  one of its
cells is of  order $\delta$. Let $\Phi_\delta : \mathbb  C \to \mathbb
C$ be  the quasi-conformal  map constructed  in the  previous section,
with Beltrami  coefficient given by~\eqref{eq:mu_f} ---  since all the
faces within a cell have  the same shape, $\Phi_\delta$ indeed depends
on $\delta$ but not on $N$.

We will re-use some of the  notation in \cite{beffara:easy}. If $z$ is
a  (triangular) face  of  $\mathcal T_{\delta,N}$,  or equivalently  a
vertex of its dual graph, let $H^{(\delta,N)}_a(z)$ be the probability
that,  for   critical  site-percolation   on  $\Omega   \cap  \mathcal
T_{\delta,N}$, there  is a chain  of pairwise distinct  black vertices
joining two boundary  vertices and separating $a$ and $z$  on one side
and $b$ and $c$ on the other. As before, define $H^{(\delta,N)}_b$ and
$H^{(\delta,N)}_c$ accordingly, and let
\begin{equation}
  H^{(\delta,N)} =  H^{(\delta,N)}_a + \tau H^{(\delta,N)}_b  + \tau^2
  H^{(\delta,N)}_c.
\end{equation}

As long as no confusion can arise,  we will drop $\delta$ and $N$ from
the notation and simply refer to $H$ and $\Phi$ where appropriate.

We  want  to   show  that  a  suitable   continuous  interpolation  of
$H^{(\delta,N)}$ is approximately quasi-conformal on $\Omega$ with the
same Beltrami  coefficient as $\Phi_\delta$.  To do that, we  will use
the characterization  in Appendix~\ref{sec:qc}; so, let  $\gamma$ be a
closed, smooth curve contained in $\Omega$, and let $\gamma_{\delta,N}
= (z_k)_{k=0..L-1}$  be a nearest-neighbor chain  of pairwise distinct
vertices of $\mathcal T_{\delta,N}^*$ which approximates $\gamma$ (for
ease  of notation,  let $z_L=z_0$).  Finally, let
\begin{equation}
  I  :=   \sum  _{k=0}   ^{L-1}  \frac   {H(z_{k+1})  +   H(z_k)}  {2}
  [\Phi(z_{k+1}) - \Phi(z_k)].
\end{equation}
This  is the  discrete  counterpart  of the  contour  integral in  the
quasi-conformal version of Morera's lemma, so  what we need to show is
that $|I|$ is small for appropriate choices of $\delta$ and $N$.

The beginning of the argument goes  the same way as for the triangular
lattice.   If   $F$  is   a   function   on  $\Omega   \cap   \mathcal
T^\ast_{\delta,N}$,  and if  $e=(\underline  e, \overline  e)$, is  an
oriented edge of $\mathcal T^\ast_{\delta,N}$, let
\begin{equation}
  F(e)   :=   \frac   {F(\underline   e)  +   F(\overline   e)}   {2};
  \qquad \partial_e F :=   F(\overline e) - F(\underline e).
\end{equation}
Then  $I$  can  be  rewritten  as  $I  =  \sum_{e\in\gamma_{\delta,N}}
H(e) \partial_e  \Phi$. Let $\mathcal  F$ be the  set of all  the dual
faces surrounded  by $\gamma_{\delta,N}$; if  $f \in \mathcal  F$, let
$\partial  f$ be  its  boundary,  read counterclockwise  as  a set  of
oriented  edges. Then,  since inner  edges  are counted  once in  each
orientation, one can rewrite
\begin{equation}
  I = \sum_{f \in \mathcal F}  \sum_{e \in \partial f} H(e) \partial_e
  \Phi.
\end{equation}

For every face  $f$ of $\mathcal T^\ast_{\delta,N}$, let  $c_f$ be the
vertex  of $\mathcal  T_{\delta,N}$ contained  in $f$.  It is  easy to
verify,  reindexing the  sums involved  above (see~\cite{beffara:easy}
for the details), that one has the identity
\begin{equation}
  I = -  \sum_{f \in \mathcal F} \sum_{e \in  \partial f} \partial_e H
  [\Phi(e) -  \Phi(c_f)].
\end{equation}
There  are two  kinds of  edges in  that sum.  For those  on the  path
$\gamma_{\delta,N}$, we know from Russo-Seymour-Welsh estimates (which
hold   due   to  Assertion   1)   that   $\partial_e  H   =   \mathcal
O((\delta/N)^\eta)$ for  some positive $\eta$; together  with the fact
that   $\Phi(e)-\Phi(c_f)  =   \mathcal   O(\delta/N)$  and   choosing
$\gamma_{\delta,N}$ with  $L = \mathcal O(N/\delta)$,  which is always
possible, the  sum over  all such  edges ends up  providing a  term of
order $(\delta/N)^\eta$. For edges inside  the curve, each term of the
form $\partial_eH \Phi(e)$ appears twice  with different sign, so they
cancel out. Denoting by $e^*$ the dual  edge of $e$ (which makes it an
edge of $\mathcal  T_{\delta,N}$), oriented so that  the rotation from
$e$ to $e^*$ goes in the positive direction, this gives
\begin{equation}
  I    =   \sum_e    \partial_eH    \partial_{e^*}\Phi   +    \mathcal
  O((\delta/N)^\eta),
\end{equation}
where the sum  ranges over all the edges  of $\mathcal T_{\delta,N}^*$
surrounded by $\gamma$.

Let  $P^{(\delta,N)}_a(e)$  be  the  probability  that  $\overline  e$
satisfies the  conditions defining $H^{(\delta,N)}_a$  but $\underline
e$ does  not; then, $\partial_e  H_a =  P_a(e) - P_a(-e)$,  where $-e$
designates the reversal  of the edge $e$, and where  again for clarity
we  drop  $(\delta,N)$  from  the   notation.  Replacing  $H$  by  its
definition, and  reordering the  terms to make  each edge  appear only
once, leads to
\begin{equation}
  I = 2  \sum_e [P_a(e) + \tau P_b(e) +  \tau^2 P_c(e)] \partial_{e^*}
  \Phi.
\end{equation}

So far, nothing is very  different from the regular triangular lattice
case, because  we are just doing  algebra. The next step  is again the
same, it uses Smirnov's ``color-switching lemma'', which can be stated
as follows.  For a  given edge $e$  of $\mathcal  T^*_{\delta,N}$, its
source $\underline  e$ has degree  $3$; denote  by $e'$ and  $e''$ the
other two edges sharing the same source, ordered so that $e$, $e'$ and
$e''$ come  in that order turning  counterclockwise around $\underline
e$. Then, the lemma is the following identity: for every edge $e$,
\begin{equation}
P_a(e)  = P_b(e') = P_c(e'').
\end{equation}
The proof is exactly  the same again as in the regular  case, so we do
not repeat it here. Replacing in the previous estimate:
\begin{equation}
  \label{eq:bigsum}
  I     =    2     \sum_e     P_a(e)     [\partial_{e^*}    \Phi     +
  \tau  \partial_{(e')^*}\Phi   +  \tau^2   \partial_{(e'')^*}\Phi]  +
  \mathcal O((\delta/N)^\eta).
\end{equation}

In the equilateral case, $\Phi$  is the identity function, the bracket
term  is identically  $0$, and  the argument  ends here.  In the  more
general case, more  work needs to be  done. The main image  to keep in
mind (although  it does  not explicitly correspond  to the  proof that
follows)  is  that  the  image   of  $\mathcal  T_{\delta,N}$  by  the
quasi-conformal  map constructed  in section~\ref{sec-2-2}  has almost
all its faces almost equilateral --- see Figure~\ref{fig:fine}.

\begin{figure}[ht!]
  \centering
  \includegraphics[width=\hsize]{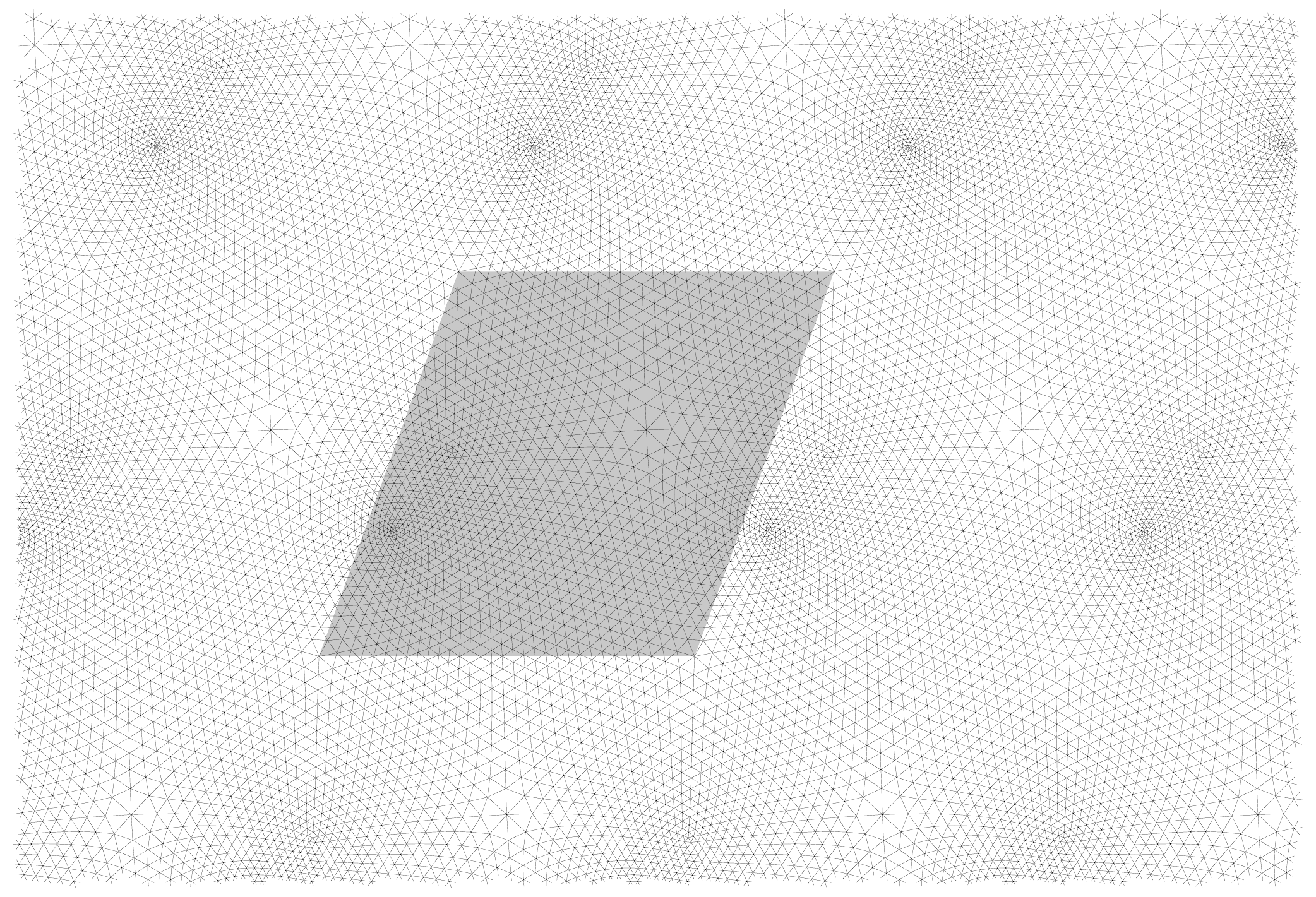}
  \caption{The image  of $\mathcal T_{\delta,N}$ by  the corresponding
    doubly periodic  quasi-conformal map; here $\delta=1$  and $N=16$.
    Away from the  initial vertices, the triangles are  close to being
    equilateral.}
  \label{fig:fine}
\end{figure}

\subsection{Controlling the bracket term}
\label{sec-4-3}

Assume for a  moment that the combinatorics of the  initial lattice is
that  of the  triangular lattice,  but  that the  embedding is  chosen
differently. More  specifically, one can  assume that $\mathcal  T$ is
the square lattice  with added diagonals in  the north-east direction.
Then  $\Phi$  is  the  real-affine  map  sending  it  to  the  regular
triangular lattice,  in other  words it maps  every face  of $\mathcal
T_{\delta,N}$ to an equilateral triangle. In that case, the bracket is
still identically equal to $0$, so $I$ is uniformly small.

What we will show is that the  general case is a small perturbation of
that  situation, as  soon as  $N$  is large  enough as  a function  of
$\delta$. Let $f$  be a (triangular) face  of $\mathcal T_{\delta,N}$,
and let $c$ be the $\delta$-cell  containing it. We first consider the
case when  $c$ (and therefore  $f$ as  well) is equilateral.  Then the
Beltrami coefficient of $\Phi$ vanishes  in $c$, in other words $\Phi$
is holomorphic in $c$.

Let  $d$ be  the  distance  between $f$  and  $\partial  c$: then  the
distorsion theorem states that the  second derivative of $\Phi$ inside
$f$  is dominated  by $\mathcal  O(|\Phi'|/d)$, in  addition to  which
$\Phi'$ is uniformly bounded on $c$. We can then do a Taylor expansion
of  the bracket  term in~\eqref{eq:bigsum}  corresponding to  $f$. The
constant term vanishes,  and so does the  first-derivative one because
the tangent map  is a complex multiplication which still  sends $f$ to
an  equilateral  triangle.  The  bracket then  reduces  to
\begin{equation}
  \mathcal O((\delta/N)^2 |\Phi''|) = \mathcal O((\delta/N)^2/d).
\end{equation}

It remains to control the sum over $c$ of that estimate. The number of
faces at distance $k\delta/N$ of the cell boundary is of order $N$, so
the  sum  over  $c$  of  the   bracket  terms  is  bounded,  up  to  a
multiplicative constant, by
\begin{equation}
  \sum_{k=1}^N N (\delta/N)^2 (N/k\delta) \simeq \delta \log N.
\end{equation}
This bound is not at all optimal, because near
the common  boundary of  two adjacent  equilateral cells,  $\Phi$ will
still  be analytic:  $d$  could be  replaced by  the  distance to  the
boundary of the  lozenge formed by these two cells.  What we will keep
from that  remark is that  the faces of $\mathcal  T_{\delta,N}$ along
the boundary  of $c$  (where $d=0$)  do not  contribute enough  to the
previous estimate  to change its  order of  magnitude. As for  the $3$
triangles near  the vertices of  $c$, the  fact that $\Phi$  is Hölder
shows that they  have a bracket of  order at most a  positive power of
$\delta/N$, and still do not contribute to the estimate above.

Actually, not much needs to be changed in the argument if the cell $c$
is  not equilateral:  one  can,  as in  the  macroscopic  case in  the
previous section,  pre-compose everything by a  real-affine map $\Psi$
of the whole plane sending $c$  to an equilateral triangle of the same
diameter. Then  $\Phi \circ \Psi^{-1}$  is analytic on  $\Psi(c)$, and
the reasoning  of the previous  paragraph applies to  it \emph{mutatis
  mutandis}.

It still  remains to  take the  sum over all  the cells  surrounded by
$\gamma$. There are of order $\delta^{-2}$ of these. Besides, the term
$P_a(e)$ is, from RSW estimates, bounded above by $(\delta/N)^{\eta'}$
for  some $\eta'>0$.  Putting  everything together,  we  obtain \[I  =
\mathcal O((\delta/N)^\eta)  + \mathcal  O(\delta^{\eta'-2} N^{-\eta'}
\delta   \log   N)   =    \mathcal   O((\delta/N)^\eta)   +   \mathcal
O(\delta^{\eta'-1} N^{-\eta'}  \log N).\] This  means that as  soon as
$N$ grows  fast enough as  a function of  $\delta$ to make  the second
error term go  to $0$, $I$ will be uniformly  controlled. The bound we
get is  not optimal at  all, but for  further reference, taking
\begin{equation}
  N = \delta ^{1-\varepsilon-1/\eta'}
\end{equation}
for arbitrary  $\varepsilon>0$  is
enough.

\begin{remark}
  The value of the $3$-arm exponent  is expected to be equal to $2/3$,
  which means that the above convergence  of $I$ to $0$ should hold as
  soon as $N \gg \delta^{-1/2}$ --- though of course we have no way to
  obtain the value of the exponent before proving conformal invariance
  in the first place.
\end{remark}

\subsection{Concluding the proof}
\label{sec-4-4}

Now that we  have the main estimate on the  discrete integral $I$, the
remainder of  the proof is  actually very  close to the  regular case.
Pick  a sequence  $\delta_k  \downarrow 0$,  choose  $N_k \to  \infty$
satisfying the previous  lower bound. Up to  a subsequence extraction,
we  can assume  that  $(H^{(\delta_k,N_k)})$  converges, uniformly  on
compact subsets of $\Omega$, to some continuous function $h:\Omega \to
\mathbb C$. Simultaneously, $\Phi_{\delta_k}$ converges uniformly to a
real-affine map $\Phi_0 : \mathbb C  \to \mathbb C$. From the previous
section,  we directly  obtain, for  any smooth  closed curve  $\gamma$
contained in $\Omega$,
\begin{equation}
  \oint_\gamma h(z)  \mathrm d \Phi_0(z) = 0.
\end{equation}
This means that  $h \circ \Phi_0^{-1}$ is holomorphic,  from which the
statement follows;  $\Phi_0$ is the  real-linear map appearing  in the
conclusion of the theorem.

\section{Concluding remarks}
\label{sec-5}

A key step of the argument  above relies on \emph{a priori} bounds for
box-crossings.  While  this looks  like  a  rather strong  assumption,
actually a closer look at the proof shows that one can almost get away
without it.  Indeed, even though  we might  not have assumption  1, we
still have  the corresponding statement  for boxes contained  within a
single $\delta$-cell because there the  graph structure is that of the
triangular lattice  (for which  box-crossing estimates are  known). In
particular, the discrete derivative estimates that we used, namely
\begin{equation}
  \partial_e H  = \mathcal  O((\delta/N)^\eta) \quad  \text{and} \quad
  P_a(e) =  \mathcal O((\delta/N)^{\eta'})
\end{equation}
can instead  be replaced by
their counterparts within  a $\delta$-cell. This leads  to much weaker
bounds:
\begin{equation}
  \partial_e H = \mathcal O(N^{-\eta}) \quad \text{and} \quad P_a(e) =
  \mathcal O(N^{-\eta'})
\end{equation}
but otherwise the proof proceeds without a change. The lower bound for
the growth of $N$ in  terms of $\delta$ grows like $\delta^{-1/\eta'}$
rather than $\delta^{1-1/\eta'}$, which is only a little worse.

The only  place where I could  not get rid  of Assumption 1 is  in the
proof of  uniform continuity for  the observable. Indeed, for  that to
hold one needs  to control differences in $H$  across different cells,
and  then  the estimate  from  the  triangular lattice  alone  becomes
trivial. On the other hand, all that  is needed here is the ability to
extract converging subsequences as  the lattice mesh vanishes; uniform
Hölder estimates  are a  nice way  to get that,  but perhaps  a weaker
version  of equicontinuity  can  be obtained  (for  instance from  the
information, which  we do have, that  there is no infinite  cluster at
the critical point).

\bigskip

One last, more  positive remark about assumption 1 is  in order: while
none  of the  known proofs  of box-crossing  estimates seems  to apply
uniformly in $(\delta,N)$ in the general  case, as soon as the initial
lattice  has  more symmetry  (for  instance,  if  in addition  to  its
periodicity it has an embedding  that is invariant under a $90$-degree
rotation) they can be extended and do provide the necessary bounds. In
that case, symmetry implies also that the modulus $\tau_G$ constructed
in the first  section is actually equal to $i$,  and the whole picture
is much more explicit.

\appendix

\section{An integral characterization of qc maps}
\label{sec:qc}

Let  $\Omega$ be  a simply  connected domain,  and $\mu  : \Omega  \to
\mathbb C$  be piecewise continuous and  such that $\|\mu\|_\infty<1$.
Fix  $\varphi_0 :  \Omega  \to  \mathbb C$  solution  to the  Beltrami
equation with coefficient $\mu$.

\begin{proposition}
  With the above notation, a continuous, injective function $\varphi :
  \Omega \to  \mathbb C$ is  itself solution to the  Beltrami equation
  with coefficient $\mu$ if and only  if for any Jordan curve $\gamma$
  in  $\Omega$, one  has  \[\oint_\gamma \varphi(z)  d \varphi_0(z)  =
  \oint_\gamma \varphi(z) \; \partial_z \varphi_0(z) \; [dz + \mu(z) d
  \bar z]= 0.\]
\end{proposition}

\begin{proof}
  Let $g = \varphi \circ \varphi_0^{-1}$. $\varphi$ is quasi-conformal
  with Beltrami coefficient  $\mu$ if and only if  $g$ is holomorphic,
  and this in turn can be  characterized using Morera's theorem: it is
  the  case if  and  only  if for  every  closed  curve $\gamma_0$  in
  $\varphi_0(\Omega)$,      \[\oint_{\gamma_0}     [\varphi      \circ
  \varphi_0^{-1}]  (w) dw  =  0.\] It  is just  a  matter of  changing
  variables, letting  $w = \varphi_0(z)$ and  $\gamma = \varphi_0^{-1}
  \circ   \gamma_0$,   to   get  \[\oint_{\gamma_0}   [\varphi   \circ
  \varphi_0^{-1}] (w)  dw = \oint_\gamma \varphi(z)  d \varphi_0(z).\]
  In more geometric terms, all the proof amounts to saying is that the
  data of $\mu$ endows $\Omega$ with a complex structure and therefore
  a  notion  of  holomorphic  function,  and that  in  a  given  chart
  ($\varphi_0$ here), those are  usual analytic function characterized
  for instance by an integral formula.
\end{proof}

\section{About the pictures, and circle packings}

Solving the Beltrami equation of section~\ref{sec-2-2} analytically is
usually impossible to do, and even though it is known that $\tau_G$ is
always an algebraic number, computing  its minimal polynomial seems to
be  beyond  the reach  of  systematic  methods. Solving  the  Beltrami
equation  numerically  is  quite  involved.  On  the  other  hand,  an
approximation  that is  good  enough for  the  purpose of  \emph{e.g.}
generating  Figure~\ref{fig:tri_qc}  can   be  obtained  using  circle
packings.

\begin{figure}[ht!]
  \centering
  \includegraphics[width=\hsize]{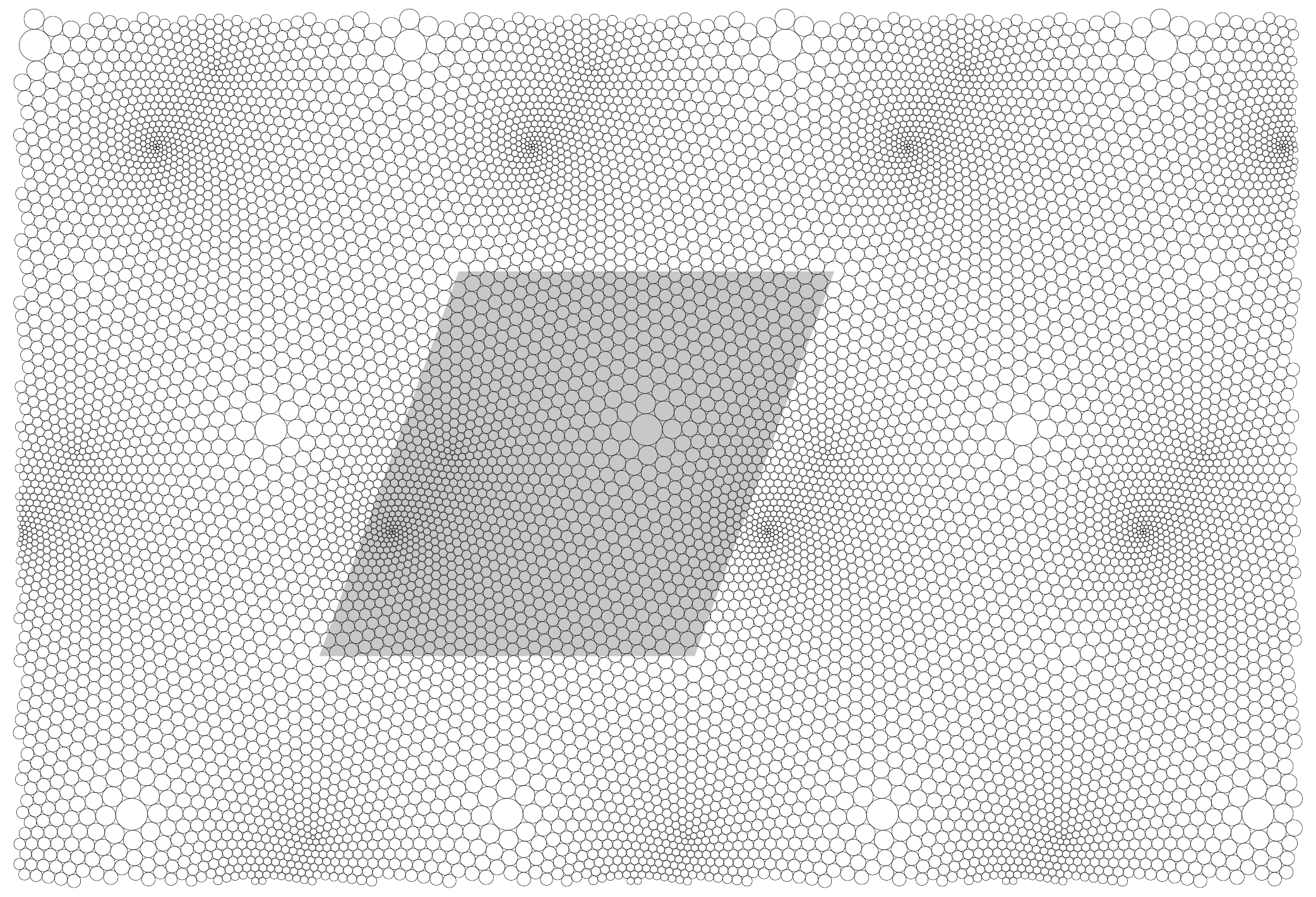}
  \caption{The circle packing obtained  from $\mathcal T_{1,N}$, where
    $\mathcal  T$ is  the  triangulation  in Figure~\ref{fig:tri}  and
    $N=16$.}
  \label{fig:cp}
\end{figure}

More specifically, keeping $\delta=1$,  to the triangulation $\mathcal
T_{1,N}$  described in  the  introduction  corresponds an  essentially
unique circle packing in the plane (see Figure~\ref{fig:cp}), which is
a collection of  disks of disjoint interiors, indexed  by the vertices
of $\mathcal T_{1,N}$, and such that two disks are tangent if and only
if  the corresponding  vertices are  adjacent (two  such packings  are
conjugated by a  global map of the form $az+b$,  and one can normalize
the choice by imposing  that $0$ and $1$ are the  centers of the disks
corresponding  to  the  vertices  that   are  at  these  locations  in
Figure~\ref{fig:tri}).

One can then re-embed $\mathcal T_{1,N}$  in the plane by mapping each
vertex  to the  center  of the  corresponding disk;  this  map can  be
extended to a piecewise linear map $\psi_N:\mathbb C \to \mathbb C$ by
interpolation  on the  faces of  $\mathcal  T_{1,N}$. In  fact, it  is
possible to show that as $N\to\infty$, $\psi_N$ converges uniformly to
the    map    $\phi_G$     constructed    in    section~\ref{sec-2-2}.
Figure~\ref{fig:tri_qc} was obtained by drawing the images by $\psi_N$
of  the edges  of $\mathcal  T_{1,\delta}$  lying along  edges of  the
initial    triangulation    $\mathcal    T$   (here    $N=16$),    and
Figure~\ref{fig:fine} by simply keeping all the edges.

\bibliographystyle{siam}
\bibliography{Biblio,Fullname}

\end{document}